\documentclass[a4paper,reqno]{amsart}
\usepackage{enumerate,amsmath,amsthm,amsfonts}
\usepackage{amsmath}
\usepackage{amsfonts}
\usepackage{amssymb}

\usepackage{epsfig}
\usepackage{graphicx,color}
\usepackage{graphpap,color}
\usepackage{graphicx}
\usepackage{latexsym}
\usepackage{float}
\usepackage{setspace}
\usepackage{lineno}   

\newtheorem{theorem} {Theorem}[section]

\newtheorem{corollary}[theorem] {Corollary}
\newtheorem{proposition} [theorem]{Proposition}

.360pk
\font\rt=cmss9.360pk
\font\sd=cmcsc9.360pk

\include {mak}
\begin{document}
~\vspace{-16mm}

\newcommand{\R}{\mathbb{R}}
\newcommand{\Z}{\mathbb{Z}}
\newcommand{\cH}{\mathcal{H}}

\newcommand{\norm}[1]{\left\|#1\right\|}
\newcommand{\hkd}[1]{\left\langle#1\right\rangle}
\newcommand{\krk}[1]{\left\{#1\right\}}
\newcommand{\krb}[1]{\left(#1\right)}
\newcommand{\abs}[1]{\left|#1\right|}
\newcommand{\ol}[1]{\overline{#1}}
\newcommand{\rref}[1]{[\ref{#1}]}

\oddsidemargin 16.5truemm
\evensidemargin 16.5truemm

\thispagestyle{plain}


\vspace{-0.25cc}

\vspace{1.2cc}

\vspace{1.5cc}

\begin{center}
{\Large\bf Centralizer Algebras of Two Permutation Groups of Order 1344 \\ 
\vspace{1.5cc}
{\large\sc M. Kosuda$^{1}$, M. Oura$^{2}$, Sarbaini$^{3}$}\\

\rule{0mm}{6mm}\renewcommand{\thefootnote}{}
}

\vspace{1.5cc}

\parbox{24cc}{{\Small{\bf Abstract.}
There are two permutation groups that they share the same character table of order $1344$. We take up natural representations on 8 and 14 letters respectively. The purpose of this paper is to examine the semi-simple structure of centralizing  algebras in the tensor representation.
}}
\end{center}
\vspace{0.25cc}

\section{INTRODUCTION} 
Let $H_1$ be a complex reflection group of order 96, which is  No. 8 in \cite{shephard}.  It is known that invariant  algebra of $H_1$ is  isomorphic to the sub algebra of modular forms for $SL_2 (\mathbb{Z})$ using theta functions \cite{broue}. Also the invariant  algebra of $H_1$ is a closely related to the  algebra of weight enumerators of self-dual and doubly even codes \cite{gleason}. In \cite{kosuda}, we took up this important group $H_1$ and the centralizer  algebras of the tensor representation of $H_1$ were determined. Additionally, in \cite{imamura}, the multi-matrix structures of the centralizer algebras of the tensor representations of a certain permutation group are discussed.

We continue this line. We take up two permutation groups of order $1344$ which have the same character table (\textit{cf}.  \cite{yoshiara}). Our groups in question are subgroups of the symmetric group of degree $8$ and $14$ . The purpose of this note is to investigate the centralizer algebras of tensor representations of their permutation representations.

As usual, let $\mathbb{C}$ denote the complex number field. We denote by $M_d$
the matrix algebra of degree $d$ over $\mathbb{C}$. For simplicity, let $nM_d$ denote $\underbrace{M_d\oplus \cdots \oplus M_d}_{\text{$n$}}$.

\section{PRELIMINARIES}
Schur-Weyl's reciprocity is one the effective methods to find the structure of the centralizer algebra of representation  $V$ of an associative algebra. Suppose that a representation $(\rho,V)$ of an associative algebra $\mathcal{A}$ is decomposed into the irreducible ones $V_i$ as follows:
\begin{equation*}
	V\cong \bigoplus_i^s m_iV_i
\end{equation*}
Here $m_i$ is the multiplicity of the simple components $V_i$ and $s$ is the number of the essential irreducible representations. The $End_A(V)$, the centralizer algebra of $A$ with respect to the representation $V$ is isomorphic to a direct sum of the endomorphism  algebras $\mathbb{C}^{m_i}$.
\begin{align*}
	End_A(V)\cong\bigoplus_i^s End_{\mathbb{C}}(\mathbb{C}^{m_i})\cong\bigoplus_i^s Mat_{m_i}(\mathbb{C})
\end{align*}
Thus to find the structure of the centralizer  algebra.
Let $G$ be a subgroup of symmetric 8, generated by 
\begin{equation*}
(5,7)(6,8),(2,3,5)(4,7,6),(1,2)(3,4)(5,6)(7,8))(1,5)(2,6)(3,7)(4,8)
\end{equation*}
On the other hand $H$ be a subgroup of symmetric 14, generated by 
\begin{align*}
(1,2,3,4,5,6)(14,13,12,11,10,9,8),(1,4,7,9,14,11,8,6)(2,5,13,10)
\end{align*}
Both $G$ and $H$ are of
order 1344. Let 
$\textbf{X}$ be the character table and
$\chi_G,\chi_H $ a permutation character. We follow the character table of \cite{yoshiara},  but switch the 7th and 8th columns. The group  $G$ has the following 11 conjugacy classes.
\begin{center}
	\begin{tabular}{l l l}
		\hline
		Class & Size & Representative\\
		\hline
		$\mathfrak{C}_1$ & $1$ & $1$\\
		$\mathfrak{C}_2$ & $7$ & $(1,2)(3,4)(5,6)(7,8)$\\
		$\mathfrak{C}_3$ & $42$ & $(1,5)(3,7)$\\
		$\mathfrak{C}_4$ & $42$ & $(1,3)(2,8)(4,6)(5,7)$\\
		$\mathfrak{C}_5$ & $84$ & $(1,5,2)(3,8,7)$\\
		$\mathfrak{C}_6$ & $168$ & $(1,2,5,6)(3,4,7,8)$\\
		$\mathfrak{C}_7$ & $168$ & $(1,5)(2,4,6,8)$\\
		$\mathfrak{C}_8$ & $224$ & $(1,2,7,4)(3,8,5,6)$\\
		$\mathfrak{C}_9$ & $224$ & $(1,7)(2,3,6,8,5,4)$\\
		$\mathfrak{C}_{10}$ & $192$ & $(2,7,4,8,6,5,3)$\\
		$\mathfrak{C}_{11}$ & $192$ & $(2,8,3,4,5,7,6)$\\
		\hline
	\end{tabular}
\end{center}
Also the group $H$ has the following 11 conjugacy classes.
\begin{center}
	\begin{tabular}{l l l}
		\hline
		Class & Size & Representative\\
		\hline
		$\mathfrak{C}'_1$ & $1$ & $1$\\
		$\mathfrak{C}'_2$ & $7$ & $(1, 14)(4, 11)(6, 9)(7, 8)$\\
		$\mathfrak{C}'_3$ & $42$ & $(1, 7, 14, 8)(4, 6, 11, 9)$\\
		$\mathfrak{C}'_4$ & $42$ & $(1, 7, 14, 8)(2, 13)(4, 9, 11, 6)(5, 10)$\\
		$\mathfrak{C}'_5$ & $84$ & $ (1, 10, 8)(3, 6, 11)(4, 12, 9)(5, 7, 14)$\\
		$\mathfrak{C}'_6$ & $168$ & $ (1, 7)(3, 12)(4, 6)(5, 10)(8, 14)(9, 11)$\\
		$\mathfrak{C}'_7$ & $168$ & $(1, 4, 7, 6, 14, 11, 8, 9)(2, 5)(3, 12)(10, 13)$\\
		$\mathfrak{C}'_8$ & $224$ & $ (1, 4, 7, 9, 14, 11, 8, 6)(2, 5, 13, 10)$\\
		$\mathfrak{C}'_9$ & $224$ & $(1, 5, 8, 14, 10, 7)(2, 13)(3, 6, 11)(4, 12, 9)$\\
		$\mathfrak{C}'_{10}$ & $192$ & $ (1, 2, 3, 4, 5, 6, 7)(8, 14, 13, 12, 11, 10, 9)$\\
		$\mathfrak{C}'_{11}$ & $192$ & $(1, 4, 7, 3, 6, 2, 5)(8, 12, 9, 13, 10, 14, 11)$\\
		\hline
	\end{tabular}
\end{center}
Suppose that $\chi_G^{\otimes k}$ and $\chi_H^{\otimes k}$are decomposed into the irreducible characters as follows:
\begin{align*}
	\chi_G^{\otimes k}=d_{G,1}^{(k)} \chi_1+\dots+d_{G,11}^{(k)} \chi_{11} \quad\text{for $G$} \\ \chi_H^{\otimes k}=d_{H,1}^{(k)} \chi_1+\dots+d_{H,11}^{(k)} \chi_{11} \quad\text{for $H$}.
\end{align*}
 We would like to find $d_G^{(k)}$ and $d_H^{(k)}$.
\section{RESULTS}
We follow the argument presented in the paper \cite{yoshiara}. It is known that  $\chi_G(g)$ (resp. $\chi_H(h)$) is  the number of elements which are fixed by $g\in G$ (resp. $h\in H$).
\begin{proposition}\label{pro1}
	We have
	\begin{align*}
		\chi_G&=\chi_1+\chi_8,  \\
		\chi_H&=\chi_1+\chi_4+\chi_7. 
	\end{align*}
\end{proposition}
\begin{proof}
	First we know the following.
\begin{center}
	\begin{tabular}{l| c c c c c c c c c c c c}
		
		& $\mathfrak{C}_1$&$\mathfrak{C}_2$&$\mathfrak{C}_6$&$\mathfrak{C}_4$&$\mathfrak{C}_3$&$\mathfrak{C}_5$&$\mathfrak{C}_9$&$\mathfrak{C}_7$&$\mathfrak{C}_8$&$\mathfrak{C}_{10}$&$\mathfrak{C}_{11}$\\
		\hline
		$\chi_{G}$ & 8& 0& 0& 0& 4& 2& 0& 2& 0& 1& 1\\
		\hline
		&$\mathfrak{C}_1'$&$\mathfrak{C}_2'$&$\mathfrak{C}_6'$&$\mathfrak{C}_3'$&$\mathfrak{C}_4'$&$\mathfrak{C}_5'$&$\mathfrak{C}_9'$&$\mathfrak{C}_7'$&$\mathfrak{C}_8'$&$\mathfrak{C}_{10}'$&$\mathfrak{C}_{11}'$\\
		\hline
		$\chi_{H}$ & 14& 6& 2& 6& 2& 2& 0& 0& 2& 0& 0	
	\end{tabular}
\end{center}
For $d_G^{(1)}=\left(d_{G,1}^{(1)}, \dots, d_{G,11}^{(1)}\right)$ we have 
\begin{equation*}
	(8, 0, 0, 0, 4, 2, 0, 2, 0, 1, 1)=\left(d_{G,1}^{(1)}, \dots, d_{G,11}^{(1)}\right)\textbf{X},
\end{equation*}
and for  $d_H^{(1)}=\left(d_{H,1}^{(1)}, \dots, d_{H,11}^{(1)}\right)$ we have 
\begin{equation*}
	(14, 6, 2, 6, 2, 2, 0, 0, 2, 0, 0)=\left(d_{H,1}^{(1)}, \dots, d_{H,11}^{(1)}\right)\textbf{X}.
\end{equation*}
Then we get
\begin{align*}
	d_G^{(1)}&=(8, 0, 0, 0, 4, 2, 0, 2, 0, 1, 1)\textbf{X}^{-1}\\
	&=(1, 0, 0, 0, 0, 0, 0, 1, 0, 0, 0)
\end{align*}
and
\begin{align*}
	d_H^{(1)}&=(14, 6, 2, 6, 2, 2, 0, 0, 2, 0, 0)\textbf{X}^{-1}\\
	&=(1, 0, 0, 1, 0, 0, 1, 0, 0, 0, 0).
\end{align*}
This completes the proof.
\end{proof}
Since $\chi_G$ is  decomposed into the trivial character and one irreducible character, $G$ is doubly transitive. On the other hand, $\chi_H$ is decomposed into the trivial character and two distinct non trivial irreducible characters. Hence $H$ is not doubly transitive,
 see  \cite{wielandt}, \cite{bannai}. Next we  calculate $d_G^{(k)}$ and $d_H^{(k)}$. 

\begin{center}
	\begin{tabular}{l| r r r r r r r r r r r r}
		\hline
		$G$& $\mathfrak{C}_1$&$\mathfrak{C}_2$&$\mathfrak{C}_6$&$\mathfrak{C}_4$&$\mathfrak{C}_3$&$\mathfrak{C}_5$&$\mathfrak{C}_9$&$\mathfrak{C}_7$&$\mathfrak{C}_8$&$\mathfrak{C}_{10}$&$\mathfrak{C}_{11}$\\
		$H$& $\mathfrak{C}_1'$&$\mathfrak{C}_2'$&$\mathfrak{C}_6'$&$\mathfrak{C}_3'$&$\mathfrak{C}_4'$&$\mathfrak{C}_5'$&$\mathfrak{C}_9'$&$\mathfrak{C}_7'$&$\mathfrak{C}_8'$&$\mathfrak{C}_{10}'$&$\mathfrak{C}_{11}'$\\
		\hline

		$\abs{C_G(g)}$ & $1344$& $192$& $16$& $32$& $32$& $6$& $6$& $8$& $8$& $7$& $7$\\
		\hline
		$\chi_1$ & $1$& $1$& $1$& $1$& $1$& $1$& $1$& $1$& $1$& $1$& $1$\\
		
		$\chi_2$ & $3$& $3$& $-1$& $-1$& $0$& $0$& $0$& $1$& $1$& $\dfrac{-1+\sqrt{-7}}{2}$& $\dfrac{-1-\sqrt{-7}}{2}$\\
		
		$\chi_3$ & $3$& $3$& $-1$& $-1$& $0$& $0$& $0$& $1$& $1$& $\dfrac{-1-\sqrt{-7}}{2}$& $\dfrac{-1+\sqrt{-7}}{2}$\\
		
		$\chi_4$ & $6$& $6$& $2$& $2$& $2$& $0$& $0$& $0$& $0$& $-1$& $-1$\\
		
		$\chi_5$ & $7$& $7$& $-1$& $-1$& $-1$& $1$& $1$& $-1$& $-1$& $0$& $0$\\
		
		$\chi_6$ & $8$& $8$& $0$& $0$& $0$& $-1$& $-1$& $0$& $0$& $1$& $1$\\
		
		$\chi_7$ & $7$& $-1$& $-1$& $3$& $-1$& $1$& $-1$& $-1$& $1$& $0$& $0$\\
		
		$\chi_8$ & $7$& $-1$& $-1$& $-1$& $3$& $1$& $-1$& $1$& $-1$& $0$& $0$\\
		
		$\chi_9$ & $14$& $-2$& $-2$& $2$& $2$& $-1$& $1$& $0$& $0$& $0$& $0$\\
		
		$\chi_{10}$ & $21$& $-3$& $1$& $1$& $-3$& $0$& $0$& $1$& $-1$& $0$& $0$\\
		
		$\chi_{11}$ & $21$& $-3$& $1$& $-3$& $1$& $0$& $0$& $-1$& $1$& $0$& $0$\\
		\hline
	\end{tabular}
\end{center}
Consider the following matrix $A_G$ such that
\begin{align*}
\begin{pmatrix}
	\chi_1\chi \\
	\chi_2\chi \\
	\vdots\\
	\chi_{11}\chi 
\end{pmatrix}&=A_G
\begin{pmatrix}
	\chi_1 \\
	\chi_2 \\
	\vdots\\
	\chi_{11} 
\end{pmatrix}.\\
\end{align*}
Then we have
\begin{align*}	
\begin{pmatrix}
	\chi_1(\mathfrak{C}_i)\chi(\mathfrak{C}_i) \\
	\chi_2(\mathfrak{C}_i)\chi(\mathfrak{C}_i) \\
	\vdots\\
	\chi_{11}(\mathfrak{C}_i)\chi(\mathfrak{C}_i) 
\end{pmatrix}&=A_G
\begin{pmatrix}
	\chi_1(\mathfrak{C}_i) \\
	\chi_2(\mathfrak{C}_i) \\
	\vdots\\
	\chi_{11}(\mathfrak{C}_i) 
\end{pmatrix},\\
\end{align*}

\begin{align*}
\begin{pmatrix}
	\chi_1(\mathfrak{C}_1)& \chi_1(\mathfrak{C}_2)&\chi_1(\mathfrak{C}_6)&\dots&\chi_{1}(\mathfrak{C}_{11}) \\
	\chi_2(\mathfrak{C}_1)&\chi_2(\mathfrak{C}_2)&\chi_2(\mathfrak{C}_6)&\dots&\chi_{2}(\mathfrak{C}_{11}) \\
	\chi_3(\mathfrak{C}_1)&\chi_3(\mathfrak{C}_2)&\chi_3(\mathfrak{C}_6)&\dots&\chi_{3}(\mathfrak{C}_{11}) \\	
	\vdots&\vdots&\vdots&\ddots&\vdots\\
	\chi_{11}(\mathfrak{C}_1)&\chi_{11}(\mathfrak{C}_2)&\chi_{11}(\mathfrak{C}_6)&\dots&\chi_{11}(\mathfrak{C}_{11}) \\ 
\end{pmatrix}
\begin{pmatrix}
	\chi (\mathfrak{C}_1)&&&&& \\
	&\chi (\mathfrak{C}_2)&&&&\\
	&&\chi (\mathfrak{C}_6)&&&\\
	&&&\chi (\mathfrak{C}_4)&&\\
	&&&&\ddots&\\
	&&&&&\chi (\mathfrak{C}_{11}) 
\end{pmatrix}&\\
=A_G
\begin{pmatrix}
	\chi_1(\mathfrak{C}_1)\dots\chi_{1}(\mathfrak{C}_{11}) \\
	\vdots\\
	\chi_{11}(\mathfrak{C}_1)\dots\chi_{11}(\mathfrak{C}_{11}) 
\end{pmatrix}.
\end{align*}
And for $A_H$ we get
\begin{align*}
\begin{pmatrix}
	\chi_1(\mathfrak{C}'_1)& \chi_1(\mathfrak{C}'_2)&\chi_1(\mathfrak{C}'_6)&\dots&\chi_{1}(\mathfrak{C}'_{11}) \\
	\chi_2(\mathfrak{C}'_1)&\chi_2(\mathfrak{C}'_2)&\chi_2(\mathfrak{C}'_6)&\dots&\chi_{2}(\mathfrak{C}'_{11}) \\
	\chi_3(\mathfrak{C}'_1)&\chi_3(\mathfrak{C}'_2)&\chi_3(\mathfrak{C}'_6)&\dots&\chi_{3}(\mathfrak{C}'_{11}) \\	
	\vdots&\vdots&\vdots&\ddots&\vdots\\
	\chi_{11}(\mathfrak{C}'_1)&\chi_{11}(\mathfrak{C'}_2)&\chi_{11}(\mathfrak{C}'_6)&\dots&\chi_{11}(\mathfrak{C}'_{11})\\
\end{pmatrix}
\begin{pmatrix}
	\chi (\mathfrak{C}'_1)&&&&& \\
	&\chi (\mathfrak{C}'_2)&&&&\\
	&&\chi (\mathfrak{C}'_6)&&&\\
	&&&\chi (\mathfrak{C}'_3)&&\\
	&&&&\ddots&\\
	&&&&&\chi (\mathfrak{C}'_{11}) 
\end{pmatrix}&\\=A_H
\begin{pmatrix}
	\chi_1(\mathfrak{C}'_1)\dots\chi_{1}(\mathfrak{C}'_{11}) \\
	\vdots\\
	\chi_{11}(\mathfrak{C}'_1)\dots\chi_{11}(\mathfrak{C}'_{11}) 
\end{pmatrix}.
\end{align*}
Thus we have \begin{align*}
	\textbf{X}diag(8, 0, 0, 0, 4, 2, 0, 2, 0, 1, 1)=A_G\textbf{X}\\
	A_G=\textbf{X}diag(8, 0, 0, 0, 4, 2, 0, 2, 0, 1, 1)\textbf{X}^{-1},
\end{align*}
and \begin{align*}
	\textbf{X}diag(14, 6, 2, 6, 2, 2, 0, 0, 2, 0, 0)=A_H\textbf{X}\\
	A_H=\textbf{X}diag(14, 6, 2, 6, 2, 2, 0, 0, 2, 0, 0)\textbf{X}^{-1}.
\end{align*}
Therefore we have 
\begin{align*}
	d_G^{(k)}&=d_G^{k-1}A_G\\
	&=d_G^{(1)}A_G^{k-1}\\
	&=d_G^{(1)}X\left(diag(8, 0, 0, 0, 4, 2, 0, 0, 2, 1, 1)\right)^{k-1}X^{-1}
\end{align*}
Subsequently, we possess
\begin{align*}
		a_{G,k}&=\dfrac{2^{3k}}{1344}+\dfrac{2^{2k}}{32}+\dfrac{7}{24}2^{k}+\dfrac{2}{7},\\ b_{G,k}&=\dfrac{2^{3k}}{448}-\dfrac{2^{2k}}{32}+\dfrac{2^{k}}{8}-\dfrac{1}{7},\\ c_{G,k}&=\dfrac{2^{3k}}{448}-\dfrac{2^{2k}}{32}+\dfrac{2^{k}}{8}-\dfrac{1}{7},\\ d_{G,k}&=\dfrac{2^{3k}}{224}+\dfrac{2^{2k}}{16}-\dfrac{2}{7},\\ e_{G,k}&=\dfrac{2^{3k}}{192}-\dfrac{2^{2k}}{32}+\dfrac{2^{k}}{24},\\
		f_{G,k}&=\dfrac{2^{3k}}{168}-\dfrac{2^k}{6}+\dfrac{2}{7},\\
		g_{G,k}&=\dfrac{2^{3k}}{192}-\dfrac{2^{2k}}{32}+\dfrac{2^{k}}{24},\\
		h_{G,k}&=\dfrac{2^{3k}}{192}+\dfrac{3}{32}2^{2k}+\dfrac{7}{24}2^{k}, \\
		i_{G,k}&=\dfrac{2^{3k}}{96}+\dfrac{2^{2k}}{16}-\dfrac{2^k}{6},\\
		j_{G,k}&=\dfrac{2^{3k}}{64}-\dfrac{3}{32}2^{2k}+\dfrac{2^{k}}{8},\\
		l_{G,k}&=\dfrac{2^{3k}}{64}+\dfrac{2^{2k}}{32}-\dfrac{2^{k}}{8}.
	\end{align*}

and for $\chi_H$
\begin{align*}
	d_H^{(k)}&=d_H^{k-1}A_H\\
	&=d_H^{(1)}A_H^{k-1}\\
	&=d_H^{(1)}X\left(diag(14, 6, 2, 6, 2, 2, 0, 2, 0, 0, 0)	\right)^{k-1}X^{-1}
\end{align*}
Afterwards
\begin{align*}
	a_{H,k}&= 2^k \left( \frac{1}{1344} \cdot 7^k + \frac{7}{192} \cdot 3^k + \frac{37}{96} \right),\\
	b_{H,k}&=2^k \left( \frac{1}{448} \cdot 7^k - \frac{1}{64} \cdot 3^k + \frac{1}{32} \right),\\
	c_{H,k}&=2^k \left( \frac{1}{448} \cdot 7^k - \frac{1}{64} \cdot 3^k + \frac{1}{32} \right), \\
	d_{H,k}&=2^k \left( \frac{1}{224} \cdot 7^k + \frac{3}{32} \cdot 3^k + \frac{3}{16} \right), \\
	e_{H,k}&=2^k \left( \frac{1}{192} \cdot 7^k + \frac{1}{192} \cdot 3^k - \frac{5}{96} \right),\\
	f_{H,k}&=2^k \left( \frac{1}{168} \cdot 7^k + \frac{1}{24} \cdot 3^k - \frac{1}{6} \right),\\
	g_{H,k}&=2^k \left( \frac{1}{192} \cdot 7^k + \frac{17}{192} \cdot 3^k + \frac{19}{96} \right),\\
	h_{H,k}&=2^k \left( \frac{1}{192} \cdot 7^k - \frac{7}{192} \cdot 3^k + \frac{7}{96} \right),\\
	i_{H,k}&=2^k \left( \frac{1}{96} \cdot 7^k + \frac{5}{96} \cdot 3^k - \frac{11}{48} \right),\\
	j_{H,k}&=2^k \left( \frac{1}{64} \cdot 7^k + \frac{1}{64} \cdot 3^k - \frac{5}{32} \right),\\
	l_{H,k}&=2^k \left( \frac{1}{64} \cdot 7^k - \frac{7}{64} \cdot 3^k + \frac{7}{32} \right).
\end{align*}	
\begin{theorem}
	We have
	\begin{align*}
		\mathfrak{A}_G^{(k)}\cong
		\begin{cases}
			2M_1\\
			M_{a_{G,k}}\oplus 2M_{b_{G,k}}\oplus M_{d_{G,k}}\oplus 2M_{e_{G,k}}\oplus M_{f_{G,k}}\oplus M_{h_{G,k}}\oplus M_{i_{G,k}}\oplus M_{j_{G,k}}\oplus M_{l_{G,k}}
		\end{cases}
	\end{align*}
and $\mathfrak{A}_H^{(k)}\cong$ $
\begin{cases}
	3M_1\\
	M_{a_{H,k}}\oplus 2M_{b_{H,k}}\oplus M_{d_{H,k}}\oplus M_{e_{H,k}}\oplus M_{f_{H,k}}\oplus M_{g_{H,k}}\oplus M_{h_{H,k}}\oplus M_{i_{H,k}}\oplus M_{j_{H,k}}\oplus M_{l_{H,k}}
\end{cases}$
where 
\begin{equation*}
	a_{G,k}, b_{G,k}, c_{G,k},d_{G,k},e_{G,k},f_{G,k},g_{G,k},h_{G,k}	,i_{G,k}, j_{G,k}, l_{G,k} 
\end{equation*} and 
\begin{equation*}
	a_{H,k}, b_{H,k}, c_{H,k},d_{H,k},e_{H,k},f_{H,k},g_{H,k},h_{H,k}	,i_{H,k}, j_{H,k}, l_{H,k}
\end{equation*} are given above.
\end{theorem}

\begin{corollary}\label{cordim}
	We have
	\begin{align*}
		\dim \mathfrak{A}_G^{(k)}&=\frac{1}{1344} \cdot 2^{6k} +  \frac{1}{32} \cdot 2^{4k} + \frac{7}{24} \cdot 2^{2k}+\frac{2}{7}\\
		\intertext{and}
			\dim \mathfrak{A}_H^{(k)}&=2^{2k} \left( \frac{1}{1344} \cdot 7^{2k} + \frac{7}{192} \cdot 3^{2k} + \frac{37}{96} \right).
	\end{align*}
\end{corollary}
The second assertion of Corollary \ref{cordim} is obtained by taking the square sum of the
dimensions of the simple components. We conclude this paper with a small
table of dim $\mathfrak{A}$.

\begin{tabular}{c| c c c c c c }
	
	&$1$&$2$&$3$&$4$&$5$&$6$\\
	\hline
	  $\dim \mathfrak{A}_G^{(k)}$& 2& 16& 342& 14606&831982&51656046	\\
	  $\dim \mathfrak{A}_H^{(k)}$ & 3& 82& 7328& 1159392&217424128&42262333952\\
	\hline
\end{tabular}

\vspace{1.5cc}
\begin{center}
ACKNOWLEDGMENTS
\end{center}
 This work was supported
 by JSPS KAKENHI Grant Numbers  24K06827 and 24K06644. The third named author of this work was supported in part by Ministry of Religious Affairs (BIB) and the
 Indonesia Endowment Fund for Education (LPDP) of the Ministry of Finance of the Republic
 of Indonesia.

\vspace{0.2cc}

 Faculty of Engineering, University of Yamanashi, 400-8511, Japan\\
\textit{email address}: mkosuda@yamanashi.ac.jp\\

Institute of Science and Engineering, Kanazawa University, Ishikawa, 920-1192, Japan\\
\textit{email address}: oura@se.kanazawa-u.ac.jp\\

Department of Mathematics, Universitas Islam Negeri Sultan Syarif Kasim Riau, Indonesia\\
\textit{and}\\

Graduate School of Natural Science and Technology, Kanazawa University, Ishikawa, 920-1192, Japan \\
\textit{email address}: sarbaini@stu.kanazawa-u.ac.jp

\end{document}